\documentclass{article}

\usepackage{amsmath}
\usepackage{amsfonts}
\usepackage{amssymb}
\usepackage{amsthm}
\usepackage[all,cmtip]{xy}
\usepackage{todonotes}
\usepackage{asymptote}
\usepackage{array}

\newcolumntype{x}[1]{%
>{\raggedleft\hspace{0pt}}p{#1}}%

\newtheorem{theorem}{Theorem}
\newtheorem{lemma}[theorem]{Lemma}

\theoremstyle{definition}
\newtheorem{definition}[theorem]{Definition}

\title{Branching Data for Algebraic Functions and Representability by Radicals}

\begin{document}

\author{Y. Burda, A. Khovanskii}

\maketitle

\begin{abstract}
The branching data of an algebraic function is a list of orders of local monodromies around branching points. We present branching data that ensure that the algebraic functions having them are representable by radicals. This paper is a review of recent work by the authors and of closely related classical work by Ritt.
\end{abstract}

\begin{center} 
\large  \textit{Dedicated to Michael Singer
on the occasion of his 60\textsuperscript{th} birthday}
\end{center}

\begin{center}
Submitted for publication to Banach Center Publications on April 1\textsuperscript{st}, 2011
\end{center}

\section{Introduction}

An algebraic function of one complex variable has a very simple invariant --- the branching orders at all its branching points. The monodromy around a branching point is the permutation induced on the sheets of the algebraic function by going around a small loop around the point. The branching order at a branching point is the order of this permutation, i.e. the smallest common multiplier of the lengths of cycles in the cycle decomposition of the monodromy around this point. This is an extremely weak invariant: its knowledge doesn't even give one the description of local monodromy around the branching points, to say nothing of the global monodromy group! Nevertheless there is a list of branching data which guarantee that algebraic functions having these branching data have a rather simple monodromy group, which in particular guarantees that these functions are representable by radicals (or solution of equation of degree 5 in one case). The class of algebraic functions with each of these branching data contains in general many functions of different degrees and having different Riemann surfaces. However all these functions turn out to be easily describable by explicit formulae. Moreover, for any branching data not from this list one can find an algebraic function with this branching having arbitrarily complicated monodromy group.
 
 \subsection{Historical remarks}

This section is written in first person by the second author, A. Khovanskii.

It was known since Jordan that the monodromy group of an algebraic function is isomorphic to the Galois group of the equation defining it over the field of rational functions. I've discovered a topological version of Galois theory giving new results about unsolvability of algebraic and differential equations in finite terms \cite{KhovanskiiTopologicalGaloisTheory},\cite{khovanskii2004solvability}. The role of Galois group in this theory is played by the monodromy group of a multivalued analytic function. While lecturing on this theory in 2006 I unexpectedly discovered \cite{Kho} some of the results presented in this article. Later I found the wonderful work of Ritt \cite{Rit} who advanced significantly further than me in many regards. However I didn't find there answers to some of the questions that interested me. I gave these questions to my graduate student Yuri Burda. Yuri has found simple and transparent answers to these questions \cite{Bur}. The central part of this article is devoted to a review of these results. They are different from those of Ritt and are obtained by different methods; however they explain essential parts of Ritt's results.

Ritt classifies rational functions of prime degree invertible in radicals. He uses the description (due to Galois) of transitive solvable permutation groups on prime number of elements.

Burda finds all branching data that guarantee that algebraic functions with these branching data are  representable by radicals. These algebraic functions can have any degree, not necessarily a prime number and the functions of prime degree described by Ritt are contained among them.

Using a different technique Ritt has also found another amazing result: a complete description of all polynomials invertible in radicals. It is hard to add anything to this description, so we just provide its exact formulation.

\subsection{Structure of The Article}

In section \ref{section:preliminaries} we define what branching datum is, remind the reader what a Galois covering is and show how to construct the minimal Galois covering dominating a given covering. In section \ref{section:class_of_functions_with_given_branching} we give an alternative description of the class of functions subject to the branching datum of a given Galois covering. In section \ref{section:galois_coverings_with_total_space_a_sphere_or_a_torus} we describe the branching data that guarantee that the total space of any Galois covering with this branching has genus zero or one. We also describe all such Galois coverings and comment on their properties. In section \ref{section:other_branching_types} we show that for other branching data the class of functions subject to them contains arbitrarily complicated coverings. Finally section \ref{section:ritt_work} explains how Ritt's results about rational functions of prime degree invertible in radicals are clarified by the results of previous sections and reviews his results about polynomials invertible in radicals.
 
\section{Preliminaries}
\label{section:preliminaries}
\subsection{Branching Data}

In this article we'll study holomorphic mappings between Riemann surfaces. Sometimes we will refer to them simply as ``branched coverings''. Of particular importance to us will be the following definition:

\begin{definition}
Let $f:X\to Y$ be a holomorphic mapping of Riemann surfaces. Let $y_0=f(x_0)$. We say that the \textit{order of branching} of $f$ at point $x_0$ is $n$ if one can find local coordinates $x$ and $y$ around $x_0$ and $y_0$ such that in the corresponding coordinate charts the mapping $f$ is given by $x\to y=x^n$.
\end{definition}

We can characterize the order of branching topologically as follows: mapping $f$ has branching of order $n$ at point $x_0$ if $f$ maps a small positively oriented loop around the point $x_0$ to $n$-th power of a small positively oriented loop around point $y_0$.

We will need the following definition later:

\begin{definition}
\textit{Branching datum} on a Riemann surface $Y$ is a finite collection of distinct points $p_1,\ldots,p_b\in Y$ and natural numbers $n_1,\ldots,n_b\geq 2$ associated to these points.

A function $f:X\to Y$ is \textit{subject to branching datum} $(p_1,n_1),\ldots,(p_b,n_b)$ if all the lifts of a small loop winding $n_i$ times around point $p_i$ to a path in $X$ are closed loops.
\end{definition}

In other words $f:X\to Y$ is subject to branching data $(p_1,n_1),\ldots,(p_b,n_b)$ if the order of branching of $f$ at any preimage of $p_i$ divides $n_i$.

We say that the branching datum $(p_1,n_1),\ldots,(p_b,n_b)$ is \textit{the branching datum of function} $f:X\to Y$ if it is the minimal branching datum to which $f$ is subject. In other words $p_1,\ldots,p_b$ are exactly the critical values of $f$ and the multiplicity $n_i$ at point $p_i$ is the least common multiple of the branching orders of $f$ at all the preimages of $p_i$.

\subsection{Galois Coverings}

\textit{The Galois group} of a branched covering $g:X\to Y$ is the group of all covering automorphisms of $g$, i.e. the set of automorphisms $h:X\to X$ fitting into the diagram
$$\xymatrix
{
X  \ar[rr]^h \ar[dr]_g & & X \ar[dl]^g \\
 & Y &
 }$$

We say that a branched covering $g:X\to Y$ is a \textit{Galois covering} if its Galois group acts transitively on all fibers. In particular the branching order of all points in a fiber over $y\in Y$ is the same. In this case we will also say that this common branching order is the branching order of the Galois covering $g$ at a point $y\in Y$ in the base.

One can think about a Galois covering with total space $X$ and Galois group $G$ simply as a quotient of $X$ by the action of the group $G$.

Given any branched covering $f:X\to Y$ one can create the minimal Galois covering $g:\tilde{X}\to Y$ that dominates it (algebraically it would correspond to taking the Galois closure of the field extension $\mathbf C(X)/\mathbf{C}(Y)$).  A loop $\gamma$ in $Y$ avoiding the critical values of $f$ lifts to a closed loop by means of $g$ if and only if its lifts based at all the points in the fiber of $f$ over the point $\gamma(0)$ are closed loops. From this description it follows that if the branching orders of points in the fiber of $f$ over $y\in Y$ are $n_1,\ldots,n_k$, then the branching order of the Galois covering $g$ over the point $y\in Y$ is the least common multiple of $n_1,\ldots,n_k$. In particular the branching datum of $f$ is the same as that of the Galois covering $g$.

\section{The Class of Functions Subject to a Given Branching Datum}
\label{section:class_of_functions_with_given_branching}

Let the branching datum $(p_1,n_1),\ldots,(p_b,n_b)$ be the branching datum of some Galois covering $g:Z\to Y$.

Consider the class of all functions $f:X\to Y$ subject to this branching datum.

We have a surprising alternative description of this class:

\begin{theorem}
\label{theorem:description_of_class_of_functions_subject_to_given_branching_datum}
A function $f:X\to Y$ is subject to the branching datum of the Galois covering $g:Z\to Y$ if and only if there exists an unbranched covering $\tilde{f}:W\to Z$ such that the covering $g\circ \tilde{f}$ dominates $f$, i.e. if and only if $f$ fits into the diagram
$$\xymatrix
{
               &  W \ar[dl]_{\tilde{g}} \ar[dr]^{\tilde{f}} &               \\
X \ar[dr]_{f}  &                                            & Z \ar[dl]^{g} \\
               &  Y                                         & 
}$$
with $\tilde{f}$ unbranched.
\end{theorem}

The proof of this theorem is very easy: if $f$ fits into such a diagram, its branching order at any point $x\in X$ must divide the branching order of $f\circ\tilde{g}$ at any preimage of $x$ under $\tilde{g}$. Since $f\circ\tilde{g}=g\circ\tilde{f}$ and $\tilde{f}$ is unbranched, this last branching order is equal to the branching order of the Galois covering $g$ over $f(x)$.

Conversely, if the branching order of $f$ is subject to the branching datum of $g$, then we can take $W$ to be an irreducible component of the fibered product of $f$ and $g$. Since the branching order of $f$ at any point $x\in X$ divides the branching order of $g$ over $f(x)$, the mapping $\tilde{f}:W\to Z$ is unbranched.

Note that the restriction that the branching datum we start with is a branching datum of a Galois covering is not essential. We can take $(p_1,n_1),\ldots,(p_b,n_b)$ to be the branching datum of any function $h:Z'\to Y$ and then take $g:Z\to Y$ to be the minimal Galois covering that dominates $h$. Its branching datum is the same as the one for $h$.

\section{Galois Coverings with Total Space a Sphere or a Torus}
\label{section:galois_coverings_with_total_space_a_sphere_or_a_torus}
\subsection{Branching Types}

We will now show all possible branching types of Galois coverings over the Riemann sphere whose total space is a Riemann surface of genus zero or one.

The main tool will be Riemann-Hurwitz formula, a relation between the Euler characteristic of the total space of a branched covering, that of the base space and the orders of branching:

Suppose that $f:X\to Y$ is a degree $n$ holomorphic mapping between Riemann surfaces $X$ and $Y$. Then $$2-2g_X=n(2-2g_Y)-\sum\limits_{x\in X} {(b_x-1)}$$
where $g_X$ and $g_Y$ are the genera of $X$ and $Y$ respectively and $b_x$ stands for the order of branching of $f$ at $x$.

In the case that $f$ is a Galois covering the branching orders at points in each fiber are the same and Riemann-Hurwitz formula can be expressed in the form
$$2-2g_X=n\left[2-2g_Y-\sum\limits_{y\in Y} {(1-\frac{1}{b_y}})\right]$$
where $b_y$ denotes now the common branching order of all points in the fiber of $f$ over $y$.

In case the genus of $X$ is known we can use the fact that the genus of $Y$ is a non-negative integer to find all the possibilities for the numbers $b_y$. For $X$ a sphere or a torus we get the following result:

\begin{theorem}
\label{theorem:branching_data_for_galois_coverings}
Let $f:X\to Y$ be a holomorphic map of Riemann surfaces which is a Galois covering. Let $n$ be the degree of $f$.

In the case $X$ is a sphere, $Y$ is also a sphere and $f$ is subject to one of the following branching data (we provide only the branching orders, the branching points can be arbitrary):

$(n,n)$

$(2,2,n)$

$(2,3,3)$,$n=12$

$(2,3,4)$,$n=24$

$(2,3,5)$,$n=60$

In the case $X$ is a torus, $Y$ is either a torus (Riemann surface of genus one) or a sphere. It is a torus if and only if $f$ is unbranched. If $Y$ is a sphere, $f$ is subject to one of the following branching data:

$(2,3,6)$

$(3,3,3)$

$(2,4,4)$

$(2,2,2,2)$
\end{theorem}

\subsection{Examples of Minimal Galois Coverings}
\label{section:examples_of_minimal_galois_coverings}
\begin{definition} A Galois covering with a branching datum $(p_1,n_1),\ldots,(p_b,n_b)$ is \textit{minimal} if it doesn't dominate any Galois covering with the same branching datum.
\end{definition}

We will now give examples of minimal Galois coverings with branching data of the types described in the previous section.

For the case that the number of branching points is at most $3$, the location of the branching points doesn't matter, so we'll refer to the branching types only by the branching orders.

Type $(n,n)$: The quotient of the sphere by the action of the cyclic group of order $n$ acting on it by means of rotations by angle $2\pi/n$ around a fixed axis.

Type $(2,2,n)$: The quotient of the sphere by the action of the dihedral group of order $2n$ acting by rotations that preserve an $n$-gon on a big circle on this sphere.

Type $(2,3,3)$: The quotient of the sphere by the action of the group of rotational symmetries of a regular tetrahedron.

Type $(2,3,4)$: The quotient of the sphere by the action of the group of rotational symmetries of a regular cube/octahedron.

Type $(2,3,5)$: The quotient of the sphere by the action of the group of rotational symmetries of a regular icosahedron/dodecahedron.

And for the torus:

Type $(2,3,6)$: The quotient by the action of the cyclic group of order 6 generated by rotation by angle $2\pi/6$ around the origin in the elliptic curve $\mathbf C/ \Lambda$, where $\Lambda$ is the lattice of Eisenstein integers $\{a+b \omega| a,b\in \mathbf Z\}$ ($\omega$ is a root of unity of order 3)

Type $(3,3,3)$: The quotient by the action of the cyclic group of order 3 generated by rotation by angle $2\pi/3$ around the origin in the elliptic curve $\mathbf C/ \Lambda$, where $\Lambda$ is the lattice of Eisenstein integers $\{a+b \omega| a,b\in \mathbf Z\}$ ($\omega$ is a root of unity of order 3)

Type $(2,4,4)$: The quotient by the action of the cyclic group of order 4 generated by rotation by angle $2\pi/4$ around the origin in the elliptic curve $\mathbf C/ \Lambda$, where $\Lambda$ is the set of Gaussian integers $\{a+b i| a,b\in \mathbf Z\}$

Type $(p_1,2),(p_2,2),(p_3,2),(p_4,2)$: The quotient by the cyclic group of order 2 generated by $z\to -z$ in an elliptic curve $\mathbf C/ \Lambda$, where $\Lambda$ is any lattice in $\mathbf{C}$. Given the points $p_1,p_2,p_3,p_4$ on the Riemann sphere the elliptic curve $\mathbf C/ \Lambda$ is reconstructed as the Riemann surface of $y=\sqrt{(x-p_1)(x-p_2)(x-p_3)(x-p_4)}$ with origin chosen at the preimage of one of the points $p_1,p_2,p_3,p_4$ (if one of them, say $p_4$, is at infinity, we take $y=\sqrt{(x-p_1)(x-p_2)(x-p_3)}$)

In fact these coverings exhaust all the possibilities for minimal Galois coverings with total space a sphere or a torus. For a sphere every Galois covering is minimal. For the torus every Galois covering with the total space a torus factors through a minimal Galois covering with total space a torus through an unbranched map. We'll discuss these results in the following two sections.
 
\subsection{The Case of a Sphere}
\label{section:case_of_a_sphere}

The results and examples described above lead to a simple proof the following classical theorem, describing Galois coverings from the Riemann sphere to the Riemann sphere, or, equivalently, finite groups of automorphisms acting on the Riemann sphere:

\begin{theorem}
\label{theorem:groups_acting_on_sphere}
Suppose that $f:S\to S$ is a Galois covering. Then it is isomorphic to one of the minimal coverings listed in section \ref{section:examples_of_minimal_galois_coverings}, i.e. to the quotient of the sphere by one of the groups below:
\begin{itemize}
\item the cyclic group of rotations by angle $2\pi/n$ around an axis
\item the group of rotational symmetries of a regular $n$-gon lying on a big circle
\item the group of rotational symmetries of a tetrahedron
\item the group of rotational symmetries of a cube (or its dual octahedron)
\item the group of rotational symmetries of an icosahedron (or its dual dodecahedron)
\end{itemize}
\end{theorem}

\begin{proof}
Since $f$ is a Galois covering of a Riemann sphere over itself, we know that its branching datum is one of those described in theorem \ref{theorem:branching_data_for_galois_coverings}. Let $g:S\to S$ be the minimal Galois covering from examples in section \ref{section:examples_of_minimal_galois_coverings} with the same branching datum. According to theorem \ref{theorem:description_of_class_of_functions_subject_to_given_branching_datum} functions $f$ and $g$ fit into the diagram
$$\xymatrix
{
               &  W \ar[dl]_{\tilde{g}} \ar[dr]^{\tilde{f}} &               \\
S \ar[dr]_{f}  &                                            & S \ar[dl]^{g} \\
               &  S                                         &  
}$$
with $\tilde{f}$ unbranched. Since every covering over a sphere which is unbranched is in fact an automorphism and $g$ is minimal we get that $f$ and $g$ fit into the diagram

$$\xymatrix
{
S \ar[dr]_{f}  \ar[rr]^{\widetilde{~~~~~}} &      & S \ar[dl]^{g} \\
      &          S                         & 
}$$
\end{proof}

\begin{theorem}
Let $f:S\to S$ be a branched covering of the Riemann sphere over itself and suppose that the total space of the minimal Galois covering that dominates it is also the Riemann sphere. Then it is dominated by one of the coverings from theorem \ref{theorem:groups_acting_on_sphere} above.
\end{theorem}

\begin{theorem}
Let $f:X\to S$  be a branched covering over the Riemann sphere and suppose that it is subject to one of the branching data $(n,n)$, $(2,2,n)$, $(2,3,3)$,$(2,3,4)$,$(2,3,5)$. Then $X$ is the Riemann sphere and the covering $f$ is dominated by one of the coverings from theorem \ref{theorem:groups_acting_on_sphere}.
\end{theorem}

The proof is the same as that of theorem \ref{theorem:groups_acting_on_sphere} (see \cite{Bur}).

\subsection{The Case of a torus}
\label{section:case_of_a_torus}

The same considerations as in the previous section let us prove the following theorem describing Galois coverings over the Riemann sphere with total space of genus one, or, equivalently, all the finite groups of automorphisms acting non-freely on a Riemann surface of genus one:

\begin{theorem}
\label{theorem:groups_acting_on_torus}
Suppose $f:T\to S$ is a Galois covering over the Riemann sphere and $T$ has genus 1. Then it is isomorphic to the quotient of the elliptic curve $\mathbf C/\Lambda$ by the action of group $G$ where
\begin{itemize}
\item $\Lambda$ is any rank 2 lattice in $\mathbb{C}$ and $G=\{z\to \pm z + b| b\in \Lambda'/\Lambda\}$ where $\Lambda'$ is a lattice containing $\Lambda$ as a sublattice of finite index.
\item $\Lambda$ is the lattice of Gaussian integers and $G=\{z\to a z + b| a^4=1, b\in \Lambda'/\Lambda\}$ where $\Lambda'$ is a lattice containing $\Lambda$ and invariant under multiplication by $i$.
\item $\Lambda$ is the lattice of Eisenstein integers and $G=\{z\to a z + b| a^3=1, b\in \Lambda'/\Lambda\}$ where $\Lambda'$ is a lattice containing $\Lambda$ and invariant under multiplication by $\omega$ ($\omega^3=1,\omega\neq 1$).
\item $\Lambda$ is the lattice of Eisenstein integers and $G=\{z\to a z + b| a^6=1, b\in \Lambda'/\Lambda\}$ where $\Lambda'$ is a lattice containing $\Lambda$ and invariant under multiplication by $\omega$ ($\omega^3=1,\omega\neq 1$).

\end{itemize} 
\end{theorem}

\begin{theorem}
Suppose that $f:X\to S$ is a branched covering such that the minimal Galois covering that dominates it has total space of genus $1$. Then $f$ is dominated by one of the coverings described in theorem \ref{theorem:groups_acting_on_torus}.
\end{theorem}

\begin{theorem}
Let $f:X\to S$  be a branched covering over the Riemann sphere and suppose that it is subject to one of the branching data $(2,3,6)$, $(3,3,3)$, $(2,4,4)$, $(2,2,2,2)$. Then $X$ has genus zero or one and the covering $f$ is dominated by one of the Galois coverings with total space of genus one described in theorem \ref{theorem:groups_acting_on_torus}.
\end{theorem}

The proofs are basically the same as the ones in the previous section. They all rely on the fact that the total space of an unbranched covering over a surface of genus one has genus one. For explicit descriptions of the coverings appearing in these theorems see \cite{Bur}.

\subsection{Remarks on Functions with Galois Covering of Genus Zero or One}
\label{section:remarks_on_special_functions}

The functions that appear in the families with branching types described in the previous two sections have some remarkable properties:

\textbf{Monodromy  Groups}: The monodromy groups of the branched coverings that appear in the two sections above are:

\textit{case (n,n) }: cyclic of order $n$ 

\textit{case (2,2,n)}: dihedral of order $2n$

\textit{case (2,3,3)}: group of even permutations on 4 elements

\textit{case (2,3,4)}:  group of all permutations on 4 elements 

\textit{case (2,3,5)}:  group of even permutations on 5 elements

\textit{cases (2,2,2,2),(3,3,3),(2,4,4),(2,3,6)}: group $\{z\to a z+b| a^k=1, b\in \Lambda'/\Lambda\}$, where $k=2,3,4$ or $6$ and $\Lambda \subset \Lambda'$ are rank 2 lattices in $\mathbf{C}$ invariant under multiplication by $a$. 

\textbf{Representability by Radicals}: {Among the groups listed above all are solvable, except the group in the case $(2,3,5)$. In particular all algebraic functions with branchings of type $(n,n)$, $(2,2,n)$, $(2,3,3)$, $(2,3,4)$, $(2,3,6)$, $(3,3,3)$, $(2,4,4)$, $(2,2,2,2)$ are representable by radicals. Algebraic functions with branching type $(2,3,5)$ are representable in terms of solution of equation of degree 5.}

\textbf{Explicit Formulae}: The polynomials in cases $(n,n)$ and $(2,2,n)$ and the rational functions in cases $(2,3,6)$,$(3,3,3)$,$(2,4,4)$,$(2,2,2,2)$ are related to division theorems for periodic or doubly periodic functions. For instance the polynomials in case $(2,2,n)$ can be identified with Chebyshev polynomials satisfying $\cos{nx}=P_n(\cos{x})$ (see also section \ref{section:invertible_polynomials}). The rational functions in the case $(2,2,2,2)$ are related to division theorem for Weierstrass $\wp$-function: they express the $\wp$-function associated to a lattice in terms of the $\wp$-function associated to its sub-lattice. This description allows us to find explicit formulae for these functions, see \cite{Bur}.
\section{Other Branching Types}
\label{section:other_branching_types}
The results of the previous two sections show that for some branching types every covering in the class of branched coverings subject to them can be easily described (e.g. the total space has genus 0 or 1, the monodromy group is one of the mentioned in section \ref{section:remarks_on_special_functions}, the functions can be described by means of explicit formulae, their inverses are expressible in terms of radicals or solution of equation of degree 5).

In this section we prove that these branching types exhaust completely all the branching types for which the class of functions subject to them is ``describable'' (for several different notions of ``describable'').

\begin{theorem}
Suppose that $\mathcal{C}$ is the class of all functions subject to a branching datum which is not of the type  $(n,n)$, $(2,2,n)$, $(2,3,3)$,$(2,3,4)$,$(2,3,5)$,$(2,3,6)$, $(3,3,3)$, $(2,4,4)$ or $(2,2,2,2)$. Then $\mathcal{C}$ is either empty, or contains
\begin{itemize}
\item Coverings with total space of arbitrarily large genus
\item Coverings with monodromy group that contains a subgroup isomorphic to $G$ for any fixed finite group $G$
\item Functions whose inverses are not expressible in $k$-radicals\footnote{An equation over a field is said to be solvable in $k$-radicals if all its roots are contained in a field obtained from the base field by a chain of extensions by radicals or solutions of equations of degree at most $k$ at each step. Solvability of an equation in $k$-radicals is completely determined by its Galois group (see \cite{KhovanskiiTopologicalGaloisTheory})} for any fixed $k$
\end{itemize}
\end{theorem}

\begin{proof}
If the class contains at least one covering in it, it also contains the minimal Galois covering $g$ dominating it. Since the branching datum of this covering is not one of the nine types mentioned in the formulation of the theorem, its total space must be of genus at least 2. By taking any unbranched covering over this total space and composing with $g$ we can get a covering in the class $\mathcal{C}$ with total space of arbitrarily large genus. This proves the first result.

Given any finite group $G$, one can find elements that generate it. In particular one gets a homomorphism from the free group on $m$ generators onto $G$ for some $m$. Now let $f:X\to S$ be a branched covering in the class $\mathcal{C}$ with total space of genus at least $m$. The fundamental group of this total space $X$ admits a homomorphism onto the free group with $m$ generators, and hence onto $G$. Corresponding to this homomorphism is an unbranched Galois covering over $X$ with monodromy group isomorphic to $G$. Composing it with the covering $f$ we get a covering in the class $\mathcal{C}$ with monodromy group containing a subgroup isomorphic to $G$.

In particular if we take $G=S_{k+1}$, $k\geq 4$, the symmetric group on $k+1$ elements, the inverse of the function corresponding to covering in class $\mathcal{C}$ whose monodromy group contains $G$ is not representable by $k$-radicals (see \cite{KhovanskiiTopologicalGaloisTheory}).
\end{proof}

\section{Ritt's Work}
\label{section:ritt_work}
\subsection{Rational Functions  of Prime Degree Invertible in Radicals}

While thinking about rational functions invertible in radicals, Ritt has found the following lemma due to Galois:

\begin{lemma}
If $G$ is a solvable group of permutations acting transitively on a prime number $p$ of elements, then the action of $G$ can be identified with the action of a subgroup of the group $\{z\to az+b|a\in(\mathbf Z/p\mathbf Z)^*, b\in \mathbf Z/p\mathbf Z \}$ on the set $\mathbf{Z}/p\mathbf{Z}$.
\end{lemma}

He then used this lemma and Riemann-Hurwitz formula to derive the following theorem:

\begin{theorem}
Let $f:S\to S$ be a rational function of prime degree $p$ whose inverse is invertible in radicals. Then it has branching type $$(p,p),(2,2,p),(2,3,6),(2,4,4),(3,3,3)\mbox{ or }( 2,2,2,2)$$ and thus the explicit descriptions of sections \ref{section:case_of_a_sphere} and \ref{section:case_of_a_torus} apply to it.
\end{theorem}

\begin{proof}
The monodromy group of such a rational function is a solvable group acting transitively on $p$ elements, hence Galois's lemma applies to it. In particular the permutations that correspond to small loops around the critical points of $f$ are of the type $z\to az+b$ for $z\in \mathbf{Z}/p\mathbf{Z}$ and $a\in (\mathbf Z/p\mathbf Z)^*, b\in \mathbf Z/p\mathbf Z$. In particular the number of preimages of such a critical point is equal to $1$ if $a=1$ and is equal to $1+\frac{p-1}{\operatorname{ord}{a}}$ otherwise ($\operatorname{ord}{a}$ is the order of $a$ in group $(\mathbf Z/p\mathbf Z)^*$). Indeed, if $a\neq 1$ there is exactly one point fixed by the permutation $z\to az+b$ and other points split into $\frac{p-1}{\operatorname{ord}{a}}$ cycles of length $\operatorname{ord}{a}$ each. Plugging this result to Riemann-Hurwitz formula gives the following restrictions on the orders $n_1,\ldots,n_k$ of the $a$'s appearing in the permutations $z\to az+b$ corresponding to small loops around critical points of $f$ (we adopt the convention that for $a=1$ the corresponding order is $\infty$):
$$2=2p-\sum\limits_{i=1}^{k}\left(p-1-\frac{p-1}{n_i}\right)$$
or $$\sum\limits_{i=1}^{k}\frac{1}{n_i}=k-2$$
($k$ is the number of critical points of $f$). There are only six solutions of this equation: \begin{align*} \frac{1}{\infty}+\frac{1}{\infty}=2-2 \\ \frac{1}{2}+\frac{1}{2}+\frac{1}{\infty}=3-2 \\ \frac{1}{2}+\frac{1}{3}+\frac{1}{6}=3-2 \\\frac{1}{2}+\frac{1}{4}+\frac{1}{4}=3-2 \\ \frac{1}{3}+\frac{1}{3}+\frac{1}{3}=3-2 \\ \frac{1}{2}+\frac{1}{2}+\frac{1}{2}+\frac{1}{2}=4-2 \end{align*}

They correspond exactly to the branching types mentioned in the theorem.
\end{proof}

\subsection{Polynomials Invertible in Radicals}
\label{section:invertible_polynomials}

We'll say that a rational function is indecomposable if it can't be written as a composition of rational functions of degrees greater than one. Ritt has noted in \cite{ritt1922prime} that every rational function can be written as a composition of indecomposable ones and that a rational function is indecomposable if and only if the monodromy of its inverse acts primitively on the sheets. Thus the classification of rational functions invertible in radicals reduces to classification of indecomposable rational functions invertible in radicals and is closely related to classification of primitive solvable groups of permutations.

This classification seems to be rather involved. However Ritt has noticed that it simplifies a lot in the case of polynomials due to the following lemma (see \cite{Rit}):

\begin{lemma}
Suppose that a primitive solvable group of permutations $G$ acts on the set of numbers from $1$ to $n$ and contains a cycle of length $n$. Then either $n$ is prime or $n=4$.
\end{lemma}

This lemma, together with the classification of rational functions of prime degree invertible in radicals and the fact that degree four polynomials are invertible in radicals gives rise to the following remarkable theorem:

\begin{theorem}
Suppose that $f:S\to S$ is a polynomial whose inverse is expressible in radicals. Then $f$ is the composition of polynomials which, up to a linear change of variable in the source and in the image, are either $z\to z^n$, Chebyshev polynomials or polynomials of degree 4.
\end{theorem}

We will show how to invert Chebyshev polynomials and polynomials of degree 4 in radicals.

For Chebyshev polynomial defined by the identity $\cos{nx}=P_n(\cos{x})$ we can find the inverse by the following trick:

\begin{align*}\cos{x}&=\frac{e^{ix}+e^{-ix}}{2}=\frac{\sqrt[n]{e^{inx}}+\sqrt[n]{e^{-inx}}}{2}\\&=\frac{\sqrt[n]{\cos{nx}+i\sin{nx}}+\sqrt[n]{\cos{nx}-i\sin{nx}}}{2}\\&=\frac{\sqrt[n]{\cos{nx}+i\sqrt{1-\cos^2{nx}}}+\sqrt[n]{\cos{nx}-i\sqrt{1-\cos^2{nx}}}}{2}\end{align*}

Thus the inverse of $w=P_n(z)$ is $$z=\frac{\sqrt[n]{w+i\sqrt{1-w^2}}+\sqrt[n]{w-i\sqrt{1-w^2}}}{2}$$

Finally to find the inverse on a degree 4 polynomial in radicals it is clearly enough to know how to solve a degree 4 polynomial equation $ax^4+bx^3+cx^2+dx+e=0$ in radicals. We will show a beautiful trick \cite{berger1987geometry} that reduces this problem to solving a cubic equation in radicals.

Instead of solving $ax^4+bx^3+cx^2+dx+e=0$ we will introduce a new variable $y=x^2$ and try to solve the system of equations $ay^2+bxy+cx^2+dx+e=0, y=x^2$. Let us give names to the quadratic forms appearing in these equations \begin{align*}Q_1(x,y)&=ay^2+bxy+cx^2+dx+e \\ Q_2(x,y)&=y-x^2\end{align*}

\begin{center}
\includegraphics[height=0.2\textheight]{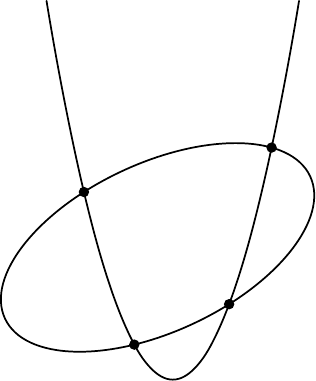}
\end{center}

Consider the pencil of all quadrics defined by equations $$Q_1(x,y)+\lambda Q_2(x,y)=0$$ In this pencil there are exactly three singular quadrics:

\begin{center}
\includegraphics[width=\the\linewidth]{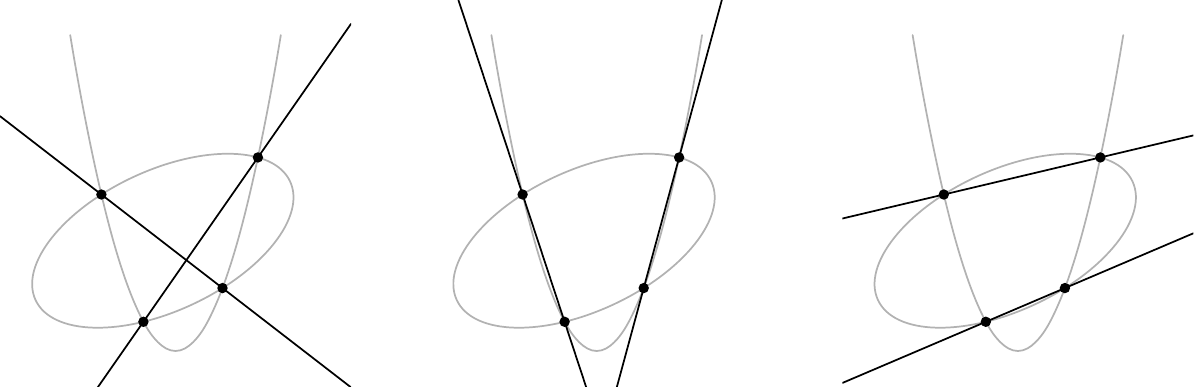}
\end{center}

To find the parameters $\lambda$ that correspond to these singular quadrics, we should solve the cubic equation in $\lambda$ equating the discriminant of the quadric from the pencil to zero. Once this equation is solved, the points of intersection of the quadrics can be found by intersecting any pair of singular quadrics we found --- this essentially amounts to simply intersecting pairs of lines.

A cubic equation that we obtained in process is in fact one we can solve already: if the cubic polynomial in question has two finite critical values, it is a Chebyshev polynomial (up to a linear change of variable). If it has only one finite critical value, it is a cube of a linear polynomial.

\bibliographystyle{plain}
\bibliography{article}

\end{document}